\documentclass[12pt]{amsproc}
\newtheorem{theorem}{\sc Theorem}[section]
\newtheorem{lemma}[theorem]{\sc Lemma}

\newtheorem{conjecture}[theorem]{\sc Conjecture}

\begin{document}

\title[Centralizers of bounded rank]
{On profinite groups in which  centralizers have bounded rank}

\author{Pavel Shumyatsky}
\address{Department of Mathematics, University of Brasilia, Brazil}
\email{pavel@unb.br}
\thanks{This research was supported by CNPq and FAPDF. The author thanks the referee for valuable comments leading to significant improvements of an earlier version of the paper.}

\maketitle

\begin{abstract}
The article deals with profinite groups in which centralizers are of finite rank. For a positive integer $r$ we prove that if $G$ is a profinite group in which the centralizer of every nontrivial element has rank at most $r$, then  $G$ is either a pro-$p$ group or a group of finite rank. Further, if $G$ is not virtually a pro-$p$ group, then $G$ is virtually of rank at most $r+1$.
\end{abstract}

\section{Introduction}

There are several recent publications dealing with profinite groups in which centralizers have certain prescribed properties (cf. \cite{shuzaza,acn,shuza,prop}). In the present paper we are concerned with the following conjecture made in \cite{shuza}.

\begin{conjecture} Let $G$ be a profinite group in which the centralizer of every nontrivial element has finite rank. Suppose that $G$ is not a pro-$p$ group. Then  $G$ has finite rank.
\end{conjecture}

Recall that a profinite group $G$ is said to have finite rank $r$ if every subgroup of $G$ can be generated by $r$ elements. Throughout the paper by a subgroup of a profinite group we mean a closed subgroup and we say that a subgroup is generated by some subset if it is topologically generated by that subset. 

Our purpose is to establish the following theorem which provides a substantial evidence in favor of the above conjecture.
\begin{theorem}\label{main} Let $r$ be a positive integer and $G$ a profinite group in which the centralizer of every nontrivial element has rank at most $r$. Then  $G$ is either a pro-$p$ group or a group of finite rank.
\end{theorem}

We mention that in a free pro-$p$ group all centralizers are procyclic and therefore pro-$p$ groups satisfying the hypothesis of Theorem \ref{main} in general need not be of finite rank.
It is natural to suspect that in the latter case of the above theorem the rank of $G$ should be bounded in terms of $r$ only. The author is pretty much doubtful about this. On the other hand, we have the following result. Recall that a group is said to virtually have some property if it has a subgroup of finite index with that property.

\begin{theorem}\label{main2} Let $G$ be a profinite group in which the centralizer of every nontrivial element has rank at most $r$. Then either $G$ is a virtually pro-$p$ group or $G$ is virtually of rank at most $r+1$.
\end{theorem}

In the next section we collect some useful results needed for the proofs of the above theorems. In Section 3 we prove Theorem \ref{main}. Section 4 contains a proof of Theorem \ref{main2}. Our notation throughout the article follows \cite{rz}.

\section{Preliminaries}

Throughout the article automorphisms of a profinite group are assumed to be continuous. If $A$ is a group of automorphisms of a group $G$, the subgroup generated by elements of the form $g^{-1}g^\alpha$ with $g\in G$ and $\alpha\in A$ is denoted by $[G,A]$. It is well-known that the subgroup $[G,A]$ is an $A$-invariant normal subgroup in $G$. We also write $C_G(A)$ for the centralizer  of $A$ in $G$.

The next lemma is a list of useful facts on coprime actions. Here $|K|$ means the order of a profinite group $K$ (see for example \cite{rz} for the definition of the order of a profinite group). For finite groups the lemma is well known (see for example \cite[Ch.~5 and 6]{go}). For infinite profinite groups the lemma follows from the finite case and the inverse limit argument (see \cite[Proposition 2.3.16]{rz} for a detailed proof of Part (iii)). As usual, $\pi(G)$ denotes the set of prime divisors of the order of $G$.

\begin{lemma}\label{cc}
Let a profinite group $A$ act by automorphisms on a profinite group $G$ such that $(|G|,|A|)=1$. Then
\begin{enumerate}
\item[(i)] $G=[G,A]C_{G}(A)$.
\item[(ii)] $[G,A,A]=[G,A]$. 
\item[(iii)] $C_{G/N}(A)=NC_G(A)/N$ for any $A$-invariant normal subgroup $N$ of $G$.
\item[(iv)] For each prime $q\in\pi(G)$ there is an $A$-invariant Sylow $q$-subgroup in $G$.
\end{enumerate}
\end{lemma}

The following theorem is immediate from the corresponding results on finite groups obtained independently by Guralnick \cite{gu} and Lucchini \cite{lu}. For finite soluble groups the corresponding result was established by Kovacs \cite{kovacs}.
\begin{theorem}\label{gulu} Let $r$ be a positive integer and $G$ a profinite group in which every Sylow subgroup has rank at most $r$. Then the rank of $G$ is at most $r+1$.
\end{theorem}

In the sequel we will need the following theorem, due to Khukhro \cite{khukhro}. We use the expression ``$(a,b,c\dots)$-bounded" to mean ``bounded from above by some function depending only on the parameters $a,b,c\dots$".

\begin{theorem}\label{khu}
Let $G$ be a finite nilpotent group with an automorphism $\alpha$ of prime order $q$ such that $C_G(\alpha)$ has rank $r$. Then $G$ contains a characteristic subgroup $N$ such that $N$ has $q$-bounded nilpotency class and $G/N$ has $(q,r)$-bounded rank.
\end{theorem}

The case $q=2$ of the above theorem was established in \cite{shu98} with a somewhat more precise statement, not even requiring the nilpotency of $G$. Using the routine inverse limit argument Theorem \ref{khu} can be extended to the case where $\alpha$ is a coprime automorphism of a profinite group $G$:
\medskip

{\it Let $G$ be a pronilpotent group admitting a coprime automorphism $\alpha$ of prime order $q$ such that $C_G(\alpha)$ has rank $r$. Then $G$ contains a characteristic subgroup $N$ such that $N$ has $q$-bounded nilpotency class and $G/N$ has $(q,r)$-bounded rank.}

\section{Proof of Theorem \ref{main}}

The next lemma deals with a crucial case of Theorem \ref{main}.
\begin{lemma}\label{nachalo} Let $r\geq1$ and $G$ a profinite group in which every nontrivial element has centralizer of rank at most $r$. Assume that $G$ has a nontrivial proper normal subgroup $M$ and a subgroup $A$ such that $(|M|,|A|)=1$ and $G=MA$. Then one of the following statements holds.
\begin{enumerate}
\item $A$ is infinite and virtually of rank at most $r$. In this case also $M$ has rank at most $r$.
\item $A$ is finite and the rank of $M$ is $(q,r)$-bounded, where $q$ is the smallest prime in $\pi(A)$.
\end{enumerate}
In either case $G$ has finite rank.
\end{lemma}
\begin{proof} If $M$ is finite the result is straightforward from Theorem \ref{gulu} so without loss of generality we assume that $M$ is infinite. Suppose first that $A$ is infinite. Let $N$ be an open subgroup of $M$ that is normal in $G$. Then $A$ induces a finite group of automorphisms of $M/N$ and therefore $A$ contains an open normal subgroup $B$ that acts on $M/N$ trivially. By Lemma \ref{cc} (iii) we have $M=NC_M(B)$. It follows that $M/N$ has rank at most $r$. Since this holds for each open subgroup of $M$ that is normal in $G$ we conclude that $M$ has rank at most $r$. Moreover, since the subgroup $B$ has nontrivial centralizer in $M$, it follows that $B$ has rank at most $r$, whence $A$ is virtually of rank at most $r$. 

Now consider the case where $A$ is finite. Choose $a\in A$ of order $q$. By Lemma \ref{cc} (iv) for each prime $p\in\pi(M)$ there is a Sylow $p$-subgroup $P$ of $M$ normalized by $a$. Observe that $a$ induces an automorphism of $P$ of order dividing $q$. Since $C_P(a)$ is of rank at most $r$, it follows from Khukhro's Theorem \ref {khu} that $P$ has a characteristic subgroup $L$ such that $L$ is nilpotent and $P/L$ has $(q,r)$-bounded rank. Since $L$ is nilpotent, it has nontrivial centre. Clearly, $L$ is contained in the centralizer of any nontrivial element of its centre. The centralizer is of rank at most $r$ and so both $L$ and $P/L$ have bounded rank. We conclude that $P$  has $(q,r)$-bounded rank. Since this holds for each prime $p\in\pi(M)$, we deduce from Theorem \ref{gulu} that $M$ has finite $(q,r)$-bounded rank, as required.
\end{proof}

\begin{proof}[Proof of Theorem \ref{main}.]
Recall that $G$ is a profinite group in which the centralizer of every nontrivial element has rank at most $r$. We wish to show that $G$ is either a pro-$p$ group or a group of finite rank. Assume that $G$ is not a pro-$p$ group. Then there exists an open normal subgroup $N$ of $G$ such that $\pi(G/N)$ contains two different primes say $q_1$ and $q_2$. Fix any prime $p\in\pi(N)$ and a Sylow $p$-subgroup $P$ of $N$. Without loss of generality we may assume $q_1\neq p$. By the Frattini argument, $G=N_G(P)N$. Therefore $N_G(P)$ contains a nontrivial $q_1$-element $a$. We make now use of Lemma 3.1 with $M=P$ and $A=\langle a\rangle$. Thus $r(P)\leq r$ or $r\leq f(r,q_1)$, where the function $f$ expresses the $(q_1,r)$-bound. Letting $d= max\{r, f(q_1,r), f(q_2,r)\}$ we conclude that every Sylow subgroup of N has rank at most $d$. Therefore, according to Theorem \ref{gulu}, $N$ has rank at most $d+1$ and since $N$ is open, $G$ has finite rank.
\end{proof}

\section{Proof of Theorem \ref{main2}}

We start this section with general facts on finite groups of given rank. Recall that the Fitting height of a finite soluble group $G$ is the length $h(G)$ of a shortest normal series of $G$ all of whose factors are nilpotent. 

\begin{lemma}\label{fitranl} Let $G$ be a finite soluble group of rank $r$. Then $h(G)$ is $r$-bounded. 
\end{lemma}
\begin{proof} Let $F=P_1\times\dots\times P_k$ be the Fitting subgroup of $G$, where $P_i$ are Sylow subgroups of $F$. Using \cite[Theorem 6.1.6]{go} we can pass to the quotient $G/\Phi(F)$ and without loss of generality assume that the subgroups $P_i$ are elementary abelian. For each $i$ the quotient $G/C_G(P_i)$ naturally embeds in the group of linear transformations of the $r$-dimensional linear space over the field with $p_i$ elements. By the well-known Zassenhaus theorem (see for example \cite[Theorem 3.23]{robinson}), the derived length of $G/C_G(P_i)$ is $r$-bounded. Let $j$ be the maximum of derived lengths of groups $G/C_G(P_i)$ for $i=1,\dots,k$. It follows that the $j$-th term of the derived series of $G$ is contained in $C_G(F)$. Taking into account that $C_G(F)\leq F$ (\cite[Theorem 6.1.3]{go}) we conclude that $G/F$ has derived length $j$. Since $j$ is $r$-bounded, the lemma is established.
\end{proof}

The nonsoluble length of a finite group $G$ is the least number of nonsoluble factors in a normal series all of whose factors are either soluble or direct products of nonabelian simple groups. We write $\lambda(G)$ for the nonsoluble length of $G$. It was shown in \cite{khu-shu131} that $\lambda(G)$ does not exceed the maximum Fitting height of soluble subgroups of $G$. Combining this with Lemma \ref{fitranl} we obtain

\begin{lemma}\label{lamranl} Let $G$ be a finite group of rank $r$. Then $\lambda(G)$ is $r$-bounded. 
\end{lemma}

Using the techniques developed by Wilson in \cite{wil83} (in particular, Lemma 2 in \cite{wil83}) the above results on finite groups of given rank can be extended to profinite groups. This enables us to deduce the following lemma.
\begin{lemma}\label{proflength} Let $G$ be a profinite group of finite rank. Then $G$ has a normal series of finite length 
$$
1=G_0\leq G_1\leq G_2\leq\dots\leq G_{s}=G \ \ \ \ \ \ \ \ \ \ \ \ \ \ \ \ \ \ \ \ (*)$$
all of whose factors are either pronilpotent or (topologically) isomorphic to Cartesian products of nonabelian finite simple groups.
\end{lemma}
We will now deduce that a profinite group of finite rank is virtually prosoluble.
\begin{lemma}\label{virtprosol} Let $G$ be a profinite group of finite rank. Then $G$ is virtually prosoluble. 
\end{lemma}
\begin{proof}
Let $$S=\prod_{i\in I}S_i$$ be a factor of the series $(*)$, where $S_i$ are nonabelian finite simple groups. Recall that the Sylow 2-subgroups of $G$ have finite rank. Since all subgroups $S_i$ have even order \cite{fetho}, it follows that $S$ is in fact a product of only finitely many finite simple groups and hence $S$ is finite. Thus, we conclude that the nonprosoluble factors of the series $(*)$ are finite. 

Let $G_{j_1+1}/G_{j_1}, G_{j_2+1}/G_{j_2},\dots, G_{j_k+1}/G_{j_k}$ be the nonprosoluble factors of the series $(*)$.
Let $$H=\{x\in G\ \vert\ \ [G_{j_i+1},x]\leq G_{j_i}\text{ for } i=1,\dots,k\}.$$ Thus $H$ is the ``centralizer" of the nonprosoluble factors of the series (*). It is straightforward that the subgroup $H$ is prosoluble. Since the factors $G_{j_i+1}/G_{j_i}$ are finite for all $i$ we deduce that $H$ is open in $G$. Hence, $G$ is virtually prosoluble, as required.
\end{proof}

We will also require the following theorem which is immediate from \cite[Theorem 5.7]{dlms}.

\begin{theorem}\label{dlms} Let $G$ be a pro-$p$ group of finite rank $r$. Then $\text{Aut}(G)$ is virtually a pro-$p$ group.
\end{theorem}

By the Fitting height of a prosoluble group $G$ we mean the length $h(G)$ of a shortest series of normal subgroups all of whose factors are pronilpotent. Note that the parameter $h(G)$ is finite if, and only if, $G$ is an inverse limit of finite soluble groups of bounded Fitting height. Lemma \ref{proflength} shows that if $G$ is a prosoluble group of finite rank, then $h(G)$ is finite. We write $F(G)$ for the maximal pronilpotent normal subgroup of $G$.

We are now ready to prove Theorem \ref{main2}.

\begin{proof}[Proof of Theorem \ref{main2}] Recall that $G$ is a profinite group in which the centralizer of every nontrivial element has rank at most $r$. We wish to show that either $G$ is virtually a pro-$p$ group for some prime $p$ or $G$ is virtually of rank at most $r+1$. Assume that $G$ is not a pro-$p$ group and so, by Theorem \ref{main}, $G$ has finite rank. By Lemma \ref{virtprosol} $G$ has an open prosoluble normal subgroup with finite Fitting height. Since we wish to prove that $G$ has a certain property virtually, without loss of generality we can assume that $G$ is prosoluble and $h=h(G)$ is finite. If $G$ is pronilpotent, then $G$ can be written as a direct product $G=K\times L$, for some subgroups $K$ and $L$ such that $(|K|,|L|)=1$. Since $K$ and $L$ centralize each other, it follows that both have rank at most $r$ and thus $G$ has rank at most $r$ as well. We therefore assume that $F(G)$ is a proper subgroup of $G$. Choose $q\in\pi(F(G))$ and let $Q_0$ be the Sylow $q$-subgroup of $F(G)$. Let $H$ be a Hall $q'$-subgroup of $G$. If $H$ is finite, then $G$ is virtually pro-$q$ and we are done. Therefore we assume that $H$ is infinite. Lemma \ref{nachalo} (applied with $M=Q_0$ and $A=H$) shows that the rank of $Q_0$ is at most $r$. Let $C=C_H(Q_0)$ In view of Theorem \ref{dlms}, $C$ is open in $H$. 
Therefore $G$ has an open normal subgroup whose intersection with $H$ is contained in $C$. We replace $G$ by that open subgroup and without loss of generality assume that $H=C$. Then for any prime $p\neq q$ the Sylow $p$-subgroups of $G$ have rank at most $r$. Since $G$ is prosoluble with finite Fitting height, $F(G)$ contains its centralizer $C_G(F(G))$. Thus, if $\pi(F(G))=\{q\}$, we obtain a contradiction since $H$ centralizes $Q_0$. Otherwise, pick a prime $p\in\pi(F(G))$ such that $p\neq q$ and let $P$ be the Sylow $p$-subgroup of $F(G)$. Choose a Sylow $q$-subgroup $Q$ of $G$. Now Lemma \ref{nachalo} (applied with $M=P$ and $A=Q$) shows that $Q$ is virtually of rank at most $r$. Let $Q_1$ be an open subgroup of $Q$ of rank $\leq r$. We replace $G$ by an open normal subgroup whose intersection with $Q$ is contained in $Q_1$. Now all Sylow subgroups of $G$ have rank at most $r$. By Theorem \ref{gulu}  $G$ has rank at most $r+1$. The proof is complete.
\end{proof}

\end{document}